\documentclass[12pt]{amsart}

\usepackage{amsmath}
\usepackage{amscd}
\usepackage{amssymb}
\usepackage{amsfonts}
\usepackage{amsthm}
\usepackage{enumerate}
\usepackage{hyperref}
\usepackage{color}
\usepackage{graphicx}

\theoremstyle{plain}
\newtheorem{thm}{Theorem}[section]

\newtheorem{prop}[thm]{Proposition}
\newtheorem{lem}[thm]{Lemma}
\newtheorem{cor}[thm]{Corollary}

\theoremstyle{definition}

\newtheorem{exmps}[thm]{Examples}

\newtheorem{dfns-rems}[thm]{Definitions and Remarks}
\newtheorem{notas-rems}[thm]{Notations and Remarks}
\newtheorem{exmps-rems}[thm]{Examples and Remarks}

\DeclareMathOperator{\depth}{depth}

\begin{document}


\title[Depth of second power of edge ideals]{Lower bounds for the depth of second power of edge ideals}


\author[S. A. Seyed Fakhari]{S. A. Seyed Fakhari}

\address{S. A. Seyed Fakhari, School of Mathematics, Statistics and Computer Science,
College of Science, University of Tehran, Tehran, Iran.}

\email{aminfakhari@ut.ac.ir}

\begin{abstract}
Assume that $G$ is a graph with edge ideal $I(G)$. We provide sharp lower bounds for the depth of $I(G)^2$ in terms of the star packing number of $G$.
\end{abstract}






\thanks{}


\maketitle


\section{Introduction} \label{sec1}

Let $\mathbb{K}$ be a field and $S = \mathbb{K}[x_1,\ldots,x_n]$  be the
polynomial ring in $n$ variables over $\mathbb{K}$. Computing and finding bounds for the depth (or equivalently, projective dimension) of homogenous ideals of $S$ and their powers have been studied in  several papers (see e.g., \cite{b'}, \cite{chhktt}, \cite{ds}, \cite{fm}, \cite{htt}, \cite{hh''}, \cite{hktt}, \cite{kty} \cite{nt}, \cite{s3}, \cite{s2}, \cite{s1} and \cite{t}).

In \cite{fhm}, Fouli, H${\rm \grave{a}}$ and Morey introduced the notion of {\it initially regular sequence} and used it to provide a method for estimating the depth of a homogenous ideal. Using this method, in \cite{fhm1}, the same authors determined a combinatorial lower bound for the depth of edge ideals of graphs. Indeed, they proved that for every graph $G$ with edge ideal $I(G)$, we have$$\depth_S I(G)\geq \alpha_2(G)+1,$$where $\alpha_2(G)$ denotes the so-called star packing number of $G$ (see Section \ref{sec2} for the definition of star packing number). Let $I(G)^{(2)}$ denote the second symbolic power of $I(G)$. In \cite{s10}, we showed that for any graph $G$,$$\depth_S I(G)^{(2)}\geq \alpha_2(G).$$It is natural to ask whether the same inequality is true if one replaces the symbolic power by ordinary. The answer is negative as we will see in Examples \ref{sharp}. However, we prove in Theorem \ref{main} that for any graph $G$, we have$$\depth_S I(G)^2\geq \alpha_2(G)-2.$$Moreover, if $G$ is a $W(K_3)$-free graph (i.e., has no induced subgraph isomorphic to a whiskered triangle), then$$\depth_S I(G)^2\geq \alpha_2(G)-1.$$Also, for any triangle-free graph $G$, we have $$\depth_S I(G)^2\geq \alpha_2(G).$$Furthermore, we provide examples showing that the above inequalities are sharp.


\section{Preliminaries and known results} \label{sec2}

In this section, we provide the definitions and the known results which will be used in the next section.

Let $G$ be a simple graph with vertex set $V(G)=\big\{x_1, \ldots,
x_n\big\}$ and edge set $E(G)$. So, we identify the vertices of $G$ with the variables of $S$. Also, by abusing the notation, every edge of $G$ will be written by the product of its vertices. For a vertex $x_i$, the {\it neighbor set} of $x_i$ is $N_G(x_i)=\{x_j\mid x_ix_j\in E(G)\}$. We set $N_G[x_i]:=N_G(x_i)\cup \{x_i\}$. The cardinality of $N_G(x_i)$ is the {\it degree} of $x_i$. A vertex of degree one is called a {\it leaf} of $G$ and the unique edge incident to a leaf is a {\it pendant edge}. For every subset $U\subset V(G)$, the graph $G\setminus U$ has vertex set $V(G\setminus U)=V(G)\setminus U$ and edge set $E(G\setminus U)=\{e\in E(G)\mid e\cap U=\emptyset\}$. A subgraph $K$ of $G$ is called {\it induced} provided that two vertices of $K$ are adjacent if and only if they are adjacent in $G$. For any graph $H$, We say that $G$ is a $H$-free graph if it has no induced subgraph isomorphic to $H$. We denote a triangle by $K_3$. A {\it whiskered triangle}, denoted by $W(K_3)$ is the graph obtained from $K_3$ by attaching a pendant edge to each of its vertices.

The {\it edge ideal} of a graph $G$ is the monomial ideal generated by quadratic squarefree monomials corresponding to the edges of $G$. In other words,$$I(G)=\big(x_ix_j \, |\,  x_ix_j\in E(G)\big)\subset S.$$

Let $G$ be a graph and $x$ be a vertex of $G$. The subgraph ${\rm St}(x)$ of $G$ with vertex set $N_G[x]$ and edge set $\{xy\, |\, y\in N_G(x)\}$ is called a {\it star with center $x$}. A {\it star packing} of $G$ is a family $\mathcal{X}$ of stars in $G$ which are pairwise disjoint, i.e., $V({\rm St}(x))\cap V({\rm St}(x'))=\emptyset$, for ${\rm St}(x), {\rm St}(x')\in \mathcal{X}$ with $x\neq x'$. The quantity$$\max\big\{|\mathcal{X}|\, |\, \mathcal{X} \ {\rm is \ a \ star \ packing \ of} \ G\big\}$$ is called the {\it star packing number} of $G$. Following \cite{fhm1}, we denote the star packing number of $G$ by $\alpha_2(G)$.

As it was mentioned in introduction, Fouli, H${\rm \grave{a}}$ and Morey \cite{fhm, fhm1} determined a lower bound for the depth of $I(G)$ in terms of the star packing number of $G$.

\begin{prop} \label{spn}
Let $G$ be graph with edge ideal $I(G)$. Then$$\depth_S I(G)\geq \alpha_2(G)+1.$$
\end{prop}

We close this section by recalling the concept of polarization.

For every monomial ideal $I$, we denote the set of minimal monomial generators of $I$ by $G(I)$.
Let $I$ be a monomial ideal of $S$ with $G(I)=u_1,\ldots,u_m$,
where $u_j=\prod_{i=1}^{n}x_i^{a_{i,j}}$, $1\leq j\leq m$. For every $i$
with $1\leq i\leq n$, let$$a_i:=\max\{a_{i,j}\mid 1\leq j\leq m\},$$and
suppose that $$T=\mathbb{K}[x_{1,1},x_{1,2},\ldots,x_{1,a_1},x_{2,1},
x_{2,2},\ldots,x_{2,a_2},\ldots,x_{n,1},x_{n,2},\ldots,x_{n,a_n}]$$ is a
polynomial ring over the field $\mathbb{K}$. Let $I^{{\rm pol}}$ be the squarefree
monomial ideal of $T$ with minimal generators $u_1^{{\rm pol}},\ldots,u_m^{{\rm pol}}$, where
$u_j^{{\rm pol}}=\prod_{i=1}^{n}\prod_{k=1}^{a_{i,j}}x_{i,k}$, $1\leq j\leq m$. The
monomial $u_j^{{\rm pol}}$ is called the {\it polarization} of $u_j$, and the ideal $I^{{\rm pol}}$
is called the {\it polarization} of $I$.


\section{Main Results} \label{sec3}

In this section, we prove the main result of this paper, Theorem \ref{main} which provides lower bounds for the depth of second power of an edge ideal $I(G)$ in terms of the star packing number of $G$. To do this, we need to prove some auxiliary lemmas. We first estimate the star packing number of a graph obtained from  $G$ by deleting a certain subset of its vertices.

\begin{lem} \label{star}
Let $G$ be a $W(K_3)$-free graph and suppose that $x_1, x_2, x_3$ are three vertices of $G$ which form a triangle. Then$$\alpha_2\big(G\setminus \bigcup_{i=1}^3N_G(x_i)\big)\geq \alpha_2(G)-2.$$
\end{lem}

\begin{proof}
To simplify the notation, set $A:=N_G(x_1)\cup N_G(x_2)\cup N_G(x_3)$. Let $\mathcal{S}$ be the set of the centers of stars in a largest star packing of $G$. In particular, $|\mathcal{S}|=\alpha_2(G)$. Since every vertex in $A$ is adjacent to at least one of the vertices $x_1, x_2, x_3$, it follows from the definition of star packing that $|\mathcal{S}\cap A|\leq 3$. If $|\mathcal{S}\cap A|\leq 2$, then the stars in $G\setminus A$ centered at the vertices in $\mathcal{S}\setminus A$ form a star packing in $G\setminus A$ of size at least $\alpha_2(G)-2$ and the assertion follows. So, we need to consider the case $|\mathcal{S}\cap A|=3$.

Let $z_1, z_2, z_3$ be the vertices belonging to $\mathcal{S}\cap A$. First, assume that$$\{z_1, z_2, z_3\}\cap \{x_1, x_2, x_3\}\neq\emptyset.$$For example, suppose that $x_1=z_1$. Then we have either $z_2\in N_G(x_2)$ or $z_2\in N_G(x_3)$. In the first case, $x_2\in N_G(z_1)\cap N_G(z_2)$ and in the second case, $x_3\in N_G(z_1)\cap N_G(z_2)$. Both are contradictions. Hence, $$\{z_1, z_2, z_3\}\cap \{x_1, x_2, x_3\}=\emptyset.$$Without loss of generality, we may assume that $z_i\in N_G(x_i)\setminus\{x_1, x_2, x_3\}$, for each integer $i=1, 2, 3$. Remind that for $i\neq j$, we have $N_G[z_i]\cap N_G[z_j]=\emptyset$. This yields that $z_iz_j, z_ix_j\notin E(G)$, for $i\neq j$. Consequently, the vertices $x_1, x_2, x_3, z_1, z_2, z_3$ form an induced $W(K_3)$ in $G$ which is a contradiction. The contradiction shows that $|\mathcal{S}\cap A|\leq 2$ which  completes the proof.
\end{proof}

The next two lemmas will be used in the proof of Corollary \ref{cordepth} which provides a lower bound for the depth of a certain ideal constructed from a graph $G$.

\begin{lem} \label{int}
Let $G$ be a graph and suppose that $x_ix_j$ is an edge of $G$. Set $L:=N_G(x_i)\cap N_G(x_j)$ and let $G'$ be the graph with $V(G')=V(G)\setminus L$ and edge set$$E(G')=E(G\setminus L)\cup\{x_px_q\mid x_p\in N_{G\setminus L}(x_i), x_q\in N_{G\setminus L}(x_j)\}.$$Then$$(I(G):x_i)\cap (I(G):x_j)=I(G')+(L).$$
\end{lem}

\begin{proof}
Observe that if $x_p\in N_{G\setminus L}(x_i)$ and $x_q\in N_{G\setminus L}(x_j)$, then $x_p\neq x_q$. We first prove the inclusion "$\supseteq$". Let $u$ be a monomial in $I(G')+(L)$. By symmetry, it is enough to prove that $u\in (I(G):x_i)$. If $u\in (L)$, then $u$ is divisible by a variable $x_r\in L$. It follows from the definition of $L$ that $x_rx_i\in I(G)$. Consequently, $ux_i\in I(G)$ which implies that $u\in (I(G):x_i)$. Therefore, assume that $u\notin (L)$. Thus, $u\in I(G')$. If $u\in I(G)$, then clearly we have $u\in (I(G):x_i)$. Hence, suppose that $u\notin I(G)$. Then we conclude from the definition of $G'$ that there are vertices $x_p\in N_{G\setminus L}(x_i)$ and $x_q\in N_{G\setminus L}(x_j)$ such that $x_px_q$ divides $u$. As $x_p\in N_G(x_i)$, we have $x_px_i\in I(G)$. Since $x_px_i$ divides $ux_i$, we deduce that $ux_i\in I(G)$. This means that $u\in (I(G):x_i)$.

To prove the reverse inclusion, let $v$ be a monomial in $(I(G):x_i)\cap (I(G):x_j)$ and suppose that $v\notin (L)$. We must show that $v\in I(G')$. If $v\in I(G)$, then we are done. Hence, assume that $v\notin I(G)$. It follows from $vx_i\in I(G)$ that there is a vertex $x_t\in N_G(x_i)$ which divide $v$. As $v\notin (L)$, we deduce that $x_t\in N_{G\setminus L}(x_i)$. Similarly, there is a vertex $x_s\in N_{G\setminus L}(x_j)$ which divide $v$. We conclude from $x_t, x_s\notin L$ that $x_t\neq x_s$. Consequently, $x_tx_s$ divides $v$. It follows from the definition of $G'$ that $x_tx_s\in I(G')$ which implies $v\in I(G')$.
\end{proof}

\begin{lem} \label{depthlem}
Let $G$ be a graph and suppose that $x_ix_j$ is an edge of $G$. Let $A$ be a subset of $N_G(x_i)\cup N_G(x_j)$ with $x_i,x_j\notin A$. Set$$J:=(I(G\setminus A):x_i)\cap (I(G\setminus A):x_j)$$and $S_A:=\mathbb{K}[x_k: 1\leq k\leq n, x_k\notin A]$. Then$$\depth_{S_A} J\geq \alpha_2(G).$$
\end{lem}

\begin{proof}
Consider the following short exact sequence.
\begin{align*}
& 0\longrightarrow \frac{S_A}{J}\longrightarrow \frac{S_A}{(I(G\setminus A):x_i)}\oplus\frac{S_A}{(I(G\setminus A):x_j)}\\ & \longrightarrow \frac{S_A}{(I(G\setminus A):x_i)+(I(G\setminus A):x_j)}\longrightarrow 0
\end{align*}
Applying depth Lemma \cite[Proposition 1.2.9]{bh} on the above exact sequence, it suffices to prove that
\begin{itemize}
\item [(a)] $\depth_{S_A}(I(G\setminus A):x_i) \geq \alpha_2(G)$,
\item [(b)] $\depth_{S_A}(I(G\setminus A):x_j) \geq \alpha_2(G)$, and
\item [(c)] $\depth_{S_A}((I(G\setminus A):x_i)+(I(G\setminus A):x_j))\geq \alpha_2(G)-1$.
\end{itemize}

To prove (a), note that
\begin{align*}
& (I(G\setminus A):x_i)=I(G\setminus(A\cup N_{G\setminus A}[x_i]))+({\rm the\ ideal\ generated\ by}\ N_{G\setminus A}(x_i))\\ & =I(G\setminus(A\cup N_G[x_i]))+({\rm the\ ideal\ generated\ by}\ N_{G\setminus A}(x_i)).
\end{align*}
Hence,
\[
\begin{array}{rl}
\depth_{S_A}(I(G\setminus A):x_i)=\depth_{S'}I(G\setminus (A\cup N_G[x_i])),
\end{array} \tag{1} \label{1}
\]
where $S'=\mathbb{K}\big[x_k: 1\leq k\leq n, x_k\notin N_{G\setminus A}(x_i)\cup A\big]$. Obviously, $x_i$ is a regular element of $S'/I(G\setminus (A\cup N_G[x_i]))$. Therefore, Proposition \ref{spn} implies that
\[
\begin{array}{rl}
\depth_{S'}I(G\setminus (A\cup N_G[x_i]))\geq \alpha_2(G\setminus (A\cup N_G[x_i]))+2.
\end{array} \tag{2} \label{2}
\]
It follows from \cite[Lemma 3.1]{s10} that
\[
\begin{array}{rl}
\alpha_2(G\setminus(A\cup N_G[x_i]))\geq \alpha_2(G)-2.
\end{array} \tag{3} \label{3}
\]
Thus, we conclude from equality (\ref{1}) and inequalities (\ref{2}) and (\ref{3}) that$$\depth_{S_A}(I(G\setminus A):x_i) \geq \alpha_2(G),$$and this completes the proof of (a). The proof of (b) is similar to that of (a). We now prove (c).

Note that
\begin{align*}
& (I(G\setminus A):x_i)+(I(G\setminus A):x_j)\\ & =I(G\setminus(A\cup N_{G\setminus A}[x_i]\cup N_{G\setminus A}[x_j]))+({\rm the\ ideal\ generated\ by}\ N_{G\setminus A}(x_i)\cup N_{G\setminus A}(x_j))\\ & =I(G\setminus(A\cup N_G[x_i]\cup N_G[x_j]))+({\rm the\ ideal\ generated\ by}\ N_{G\setminus A}(x_i)\cup N_{G\setminus A}(x_j))\\ & =I(G\setminus(N_G[x_i]\cup N_G[x_j]))+({\rm the\ ideal\ generated\ by}\ N_{G\setminus A}(x_i)\cup N_{G\setminus A}(x_j)),
\end{align*}
where the last equality follows from $A\subseteq N_G(x_i)\cup N_G(x_j)$. We conclude that
\[
\begin{array}{rl}
\depth_{S_A}((I(G\setminus A):x_i)+(I(G\setminus A):x_j))=\depth_{S''}I(G\setminus(N_G[x_i]\cup N_G[x_j])),
\end{array} \tag{4} \label{4}
\]
where $S''=\mathbb{K}\big[x_k: 1\leq k\leq n, x_k\notin N_{G\setminus A}(x_i)\cup N_{G\setminus A}(x_j)\cup A\big]$. Using Proposition \ref{spn}, we deuce that
\[
\begin{array}{rl}
\depth_{S''}I(G\setminus(N_G[x_i]\cup N_G[x_j]))\geq \alpha_2(G\setminus(N_G[x_i]\cup N_G[x_j]))+1.
\end{array} \tag{5} \label{5}
\]
We also know from \cite[Lemma 3.1]{s10} that
\[
\begin{array}{rl}
\alpha_2(G\setminus(N_G[x_i]\cup N_G[x_j]))\geq \alpha_2(G)-2.
\end{array} \tag{6} \label{6}
\]
Consequently, the assertion of (c) follows from equality (\ref{4}) and inequalities (\ref{5}) and (\ref{6}).
\end{proof}

The following corollary is a consequence of Lemmata \ref{int} and \ref{depthlem}.

\begin{cor} \label{cordepth}
Assume that $G$ is a graph and $x_ix_j$ is an edge of $G$. Let $A$ be a subset of $N_G(x_i)\cup N_G(x_j)$ with $x_i,x_j\notin A$. Set $L:=N_{G\setminus A}(x_i)\cap N_{G\setminus A}(x_j)$ and $S_A:=\mathbb{K}[x_k: 1\leq k\leq n, x_k\notin A]$. Suppose $G'$ is the graph with $V(G')=V(G)\setminus (A\cup L)$ and edge set$$E(G')=E\big(G\setminus (A\cup L)\big)\cup\big\{x_px_q\mid x_p\in N_{G\setminus(A\cup L)}(x_i), x_q\in N_{G\setminus(A\cup L)}(x_j)\big\}.$$ Then$$\depth_{S_A}\big(I(G')+(L)\big)\geq \alpha_2(G).$$
\end{cor}

\begin{proof}
Let $J$ be the ideal defined in Lemma \ref{depthlem}. By substituting $G$ with $G\setminus A$ in Lemma \ref{int}, we obtain that $J=I(G')+(L)$. The claim now follows from Lemma \ref{depthlem}
\end{proof}

The following lemma is the most technical part of the proof of Theorem \ref{main}.

\begin{lem} \label{last}
Let $G$ be a graph and suppose that $x_ix_j$ is an edge of $G$. Let $A$ be a subset of $N_G(x_i)\cup N_G(x_j)$ with $x_i,x_j\notin A$. Then$$\depth_{S_A} \big(I(G\setminus A)^2:x_ix_j\big)\geq\alpha_2(G)-2.$$If moreover, $G$ is a $W(K_3)$-free graph, then$$\depth_{S_A}\big(I(G\setminus A)^2:x_ix_j\big)\geq\alpha_2(G)-1.$$
\end{lem}

\begin{proof}
Set $r:=|(N_G(x_i)\cup N_G(x_j))\setminus \{x_i,x_j\}|$. Then $|A|\leq r$. We proceed by backward induction on $|A|$. If $|A|=r$, then $A=(N_G(x_i)\cup N_G(x_j))\setminus \{x_i,x_j\}$. Hence, $G\setminus A$ is the disjoint union of the edge $x_ix_j$ with the graph $G\setminus (N_G(x_i)\cup N_G(x_j))$. Thus, we conclude from \cite[Lemma 2.10]{m} that$$\big(I(G\setminus A)^2:x_ix_j\big)=I(G\setminus A).$$Consequently, we deduce from Proposition \ref{spn}  that$$\depth_{S_A}\big(I(G\setminus A)^2:x_ix_j\big)\geq \alpha_2(G\setminus A)+1.$$On the other hand, we know from \cite[Lemma 3.1]{s10} that $\alpha_2(G\setminus A)\geq \alpha_2(G)-2$. This together with the above inequality implies that$$\depth_{S_A} \big(I(G\setminus A)^2:x_ix_j\big)\geq\alpha_2(G)-1.$$Therefore, assume that $|A|\leq r-1$.

Set $L:=N_{G\setminus A}(x_i)\cap N_{G\setminus A}(x_j)$ and let $G'$ be the graph introduced in Corollary \ref{cordepth}. We know from \cite[Theorems 6.5 and 6.7]{b} that
\begin{align*}
& \big(I(G\setminus A)^2:x_ix_j\big)=\\ & I(G\setminus A)+(x_px_q\mid x_p\in N_{G\setminus A}(x_i), x_q\in N_{G\setminus A}(x_j))=\\ & I(G\setminus A)+(x_px_q\mid x_p\in N_{G\setminus A}(x_i), x_q\in N_{G\setminus A}(x_j), x_p\neq x_q)+(x_k^2: x_k\in L).
\end{align*}
If $L=\emptyset$, then using the above equalities, we have$$\big(I(G\setminus A)^2:x_ix_j\big)=I(G')=I(G')+(L).$$Thus, in this case the assertion follows from Corollary \ref{cordepth}. Hence, suppose that $L\neq\emptyset$. Without loss of generality, assume that $L=\{x_1, \ldots, x_t\}$, for some integer $t\geq 1$. Let $H$ be the graph with$$I(H)=\big(I(G\setminus A)^2:x_ix_j\big)^{\rm pol}.$$In other words, $H$ is the graph with vertex set $V(H)=V(G\setminus A)\cup\{y_1, \ldots, y_t\}$ and edge set$$E(H)=E(G\setminus A)\cup \{x_px_q\mid x_p\in N_{G\setminus A}(x_i), x_q\in N_{G\setminus A}(x_j), x_p\neq x_q\}\cup\{x_1y_1, \ldots, x_ty_t\}.$$Let $T$ be the polynomial ring over $\mathbb{K}$ with variables corresponding to the vertices of $H$. It follows from \cite[Corollary 1.6.3]{hh} that
\[
\begin{array}{rl}
\depth_{S_A} \big(I(G\setminus A)^2:x_ix_j\big)=\depth_T I(H)-t.
\end{array} \tag{7} \label{7}
\]
Consider the short exact sequence
\begin{align*}
0\longrightarrow \frac{T}{(I(H):x_1)}\longrightarrow \frac{T}{I(H)} \longrightarrow \frac{T}{I(H)+(x_1)}\longrightarrow 0.
\end{align*}
It follows from depth Lemma \cite[Proposition 1.2.9]{bh} that
\[
\begin{array}{rl}
\depth_T I(H) \geq \min\big\{\depth_T (I(H):x_1), \depth_T (I(H),x_1)\big\}.
\end{array} \tag{8} \label{8}
\]
Therefore, using equality (\ref{7}) and inequality (\ref{8}) it is enough to prove the following statements.
\begin{itemize}
\item[(i)] $\depth_T (I(H):x_1)$ is at least $\alpha_2(G)+t-2$, and if $G$ is a $W(K_3)$-free graph, then $\depth_T (I(H):x_1)\geq \alpha_2(G)+t-1$.

\item[(ii)] $\depth_T (I(H), x_1)$ is at least $\alpha_2(G)+t-2$, and if $G$ is a $W(K_3)$-free graph, then $\depth_T (I(H),x_1)\geq \alpha_2(G)+t-1$.
\end{itemize}

We first prove (i). Note that$$(I(H):x_1)=I(H\setminus N_H[x_1])+({\rm the \ ideal \ generated \ by} \ N_H(x_1)).$$
Consequently,
\[
\begin{array}{rl}
\depth_T (I(H):x_1)=\depth_{T'} I(H\setminus N_H[x_1]),
\end{array} \tag{9} \label{9}
\]
where $T'$ is the polynomial ring which is obtained from $T$ by deleting the variables in $N_H(x_1)$. It is obvious that$$N_H[x_1]=N_{G\setminus A}(x_1)\cup N_{G\setminus A}(x_i)\cup N_{G\setminus A}(x_j)\cup\{y_1\}.$$In particular, the vertices $x_2, \ldots, x_t$ are contained in $N_H[x_1]$. Thus, $H\setminus N_H[x_1]$ is disjoint union of the isolated vertices $y_2, \ldots, y_t$ with the graph $H'$ defined as
\begin{align*}
H':= & G\setminus\big(A\cup N_{G\setminus A}(x_1)\cup N_{G\setminus A}(x_i)\cup N_{G\setminus A}(x_j)\big)\\ & =G\setminus \big(A\cup N_G(x_1)\cup N_G(x_i)\cup N_G(x_j)\big)\\ & =G\setminus\big(N_G(x_1)\cup N_G(x_i)\cup N_G(x_j)\big),
\end{align*}
where the last equality follows from $A\subseteq N_G(x_i)\cup N_G(x_j)$. Since $x_1, y_2, \ldots, y_t$ is a regular sequence on $T'/I(H\setminus N_H[x_1])$, we deduce that
\[
\begin{array}{rl}
\depth_{T'} I(H\setminus N_H[x_1])=\depth_{T''} I(H')+t,
\end{array} \tag{10} \label{10}
\]
where $T''$ is the polynomial ring which is obtained from $T'$ by deleting the variables $x_1, y_2, \ldots y_t$. Using Proposition \ref{spn}, we have
\[
\begin{array}{rl}
\depth_{T''} I(H')\geq\alpha_2(H')+1.
\end{array} \tag{11} \label{11}
\]
Also, we conclude from \cite[Lemma 3.1]{s10} and Lemma \ref{star} that
$$\alpha_2(H')=\alpha_2\big(G\setminus(N_G(x_1)\cup N_G(x_i)\cup N_G(x_j))\big)\geq \alpha_2(G)-3$$and if $G$ is a $W(K_3)$-free graph, then$$\alpha_2(H')=\alpha_2\big(G\setminus(N_G(x_1)\cup N_G(x_i)\cup N_G(x_j))\big)\geq \alpha_2(G)-2.$$
Together with equalities (\ref{9}) and (\ref{10}) and inequality (\ref{11}), we obtain the assertion of (i).

To prove (ii), note that$$(I(H),x_1)=I(H\setminus x_1)+(x_1).$$Therefore,
\[
\begin{array}{rl}
\depth_T(I(H),x_1)=\depth_{T_1}I(H\setminus x_1),
\end{array} \tag{12} \label{12}
\]
where $T_1$ is the polynomial ring obtained from $T$ by deleting the variable $x_1$. Since $y_1$ is an isolated vertex of $H\setminus x_1$, we conclude that $y_1$ is regular on $T_1/I(H\setminus x_1)$. Hence, equality (\ref{12}) yields that
\[
\begin{array}{rl}
\depth_T(I(H),x_1)=\depth_{T_2}I(H\setminus \{x_1, y_1\})+1,
\end{array} \tag{13} \label{13}
\]
where $T_2$ is the polynomial ring obtained from $T_1$ by deleting the variable $y_1$. Set $A':=A\cup\{x_1\}$. As $x_1\in L$, we have $A'\subseteq N_{G}(x_i)\cup N_{G}(x_j)$. Clearly, $x_i$ and $x_j$ do not belong to $A'$. Set $S_{A'}:=\mathbb{K}[x_k: 1\leq k\leq n, x_k\notin A']$. Since $|A'| > |A|$, it follows from the induction hypothesis that
\[
\begin{array}{rl}
\depth_{S_{A'}} \big(I(G\setminus A')^2:x_ix_j\big)\geq\alpha_2(G)-2,
\end{array} \tag{14} \label{14}
\]
and if moreover $G$ is a $W(K_3)$-free graph, then
\[
\begin{array}{rl}
\depth_{S_{A'}}\big(I(G\setminus A')^2:x_ix_j\big)\geq\alpha_2(G)-1.
\end{array} \tag{15} \label{15}
\]
It is obvious that$$\big(I(G\setminus A')^2:x_ix_j\big)^{\rm pol}=I(H\setminus\{x_1, y_1\}).$$ Thus, Using \cite[Corollary 1.6.3]{hh}, we have
\[
\begin{array}{rl}
\depth_{S_{A'}}\big(I(G\setminus A')^2:x_ix_j\big)=\depth_{T_2} I(H\setminus\{x_, y_1\})-(t-1).
\end{array} \tag{16} \label{16}
\]
Combining equalities (\ref{13}) and (\ref{16}) and inequalities (\ref{14}) and (\ref{15}) completes the proof of (ii).
\end{proof}

We are now ready to prove the main result of this paper.

\begin{thm} \label{main}
\begin{itemize}
\item[(1)] For any graph $G$, we have $\depth_S I(G)^2\geq \alpha_2(G)-2$.
\item[(2)] For any $W(K_3)$-free graph $G$, we have $\depth_S I(G)^2\geq \alpha_2(G)-1$.
\item[(3)] For any triangle-free graph $G$, we have $\depth_S I(G)^2\geq \alpha_2(G)$.
\end{itemize}
\end{thm}

\begin{proof}
Set $I:=I(G)$ and let $G(I)=\{u_1, \ldots, u_m\}$ be the set of minimal monomial generators of $I$. For every integer $k$ with $1\leq k\leq m$, consider the short exact sequence
\begin{align*}
0 & \longrightarrow \frac{S}{(I^2+(u_1, \ldots, u_{k-1})):u_k}\longrightarrow \frac{S}{I^2+(u_1, \ldots, u_{k-1})}\\ & \longrightarrow \frac{S}{I^2+(u_1, \ldots, u_k)}\longrightarrow 0,
\end{align*}
where for $k=1$, the ideal $(u_1, \ldots, u_{k-1})$ is the zero ideal. It follows from depth Lemma \cite[Proposition 1.2.9]{bh} that
\begin{align*}
& \depth_S(I^2+(u_1, \ldots, u_{k-1}))\\ & \geq \min\big\{\depth_S((I^2+(u_1, \ldots, u_{k-1})):u_k), \depth_S(I^2+(u_1, \ldots, u_k))\big\}.
\end{align*}
Using the above inequality inductively, we have
\begin{align*}
& \depth_S I^2\geq\\ & \min\bigg\{\depth_S (I^2+I), \min\big\{\depth_S ((I^2+(u_1, \ldots, u_{k-1})):u_k)\mid 1\leq k\leq m\big\}\bigg\}= \\ & \min\big\{\depth_S I, \depth_S((I^2+(u_1, \ldots, u_{k-1})):u_k)\mid 1\leq k\leq m\big\}\geq\\ & \min\big\{\alpha_2(G)+1, \depth_S((I^2+(u_1, \ldots, u_{k-1})):u_k)\mid 1\leq k\leq m\big\},
\end{align*}
where the last inequality follows from Proposition \ref{spn}. Hence, in order to prove (1) and (2) it is enough to show that for every integer $k$ with $1\leq k\leq m$, we have$$\depth_S((I^2+(u_1, \ldots, u_{k-1})):u_k)\geq \alpha_2(G)-2,$$and if $G$ is a $W(K_3)$-free graph, then$$\depth_S((I^2+(u_1, \ldots, u_{k-1})):u_k)\geq \alpha_2(G)-1.$$

Fix an integer $k$ with $1\leq k\leq m$ and assume that $u_k=x_ix_j$. Using \cite[Theorem 4.12]{b}, we may suppose that for every pair of integers $1\leq s< t\leq m$, one of the following conditions holds.
\begin{itemize}
\item [(i)] $(u_s:u_t) \subseteq (I^2:u_t)$; or
\item [(ii)] there exists an integer $\ell\leq t-1$ such that $(u_{\ell}:u_t)$ is generated by a variable, and $(u_s:u_t)\subseteq (u_{\ell}:u_t)$.
\end{itemize}
We conclude from (i) and (ii) above that
\[
\begin{array}{rl}
\big((I^2+ (u_1, \ldots, u_{k-1})):u_k\big)=(I^2:u_k)+({\rm some \ variables}).
\end{array} \tag{17} \label{17}
\]
Assume that $A$ is the set of variables belonging to $\big((I^2+ (u_1, \ldots, u_{k-1})):u_k\big)$. Let $x_r$ be an arbitrary variable in $A$. This means that $x_rx_ix_j$ belongs to the ideal $I^2+(u_1, \ldots, u_{k-1})$. As $I^2$ is generated in degree $4$, we deduce that $x_rx_ix_j\notin I^2$. Hence, there is an integer $l$ with $1\leq l\leq k-1$ such that $u_l$ divides $x_rx_ix_j$. Since, $u_l\neq u_k$, it follows that either $u_l=x_rx_i$ or $u_l=x_rx_j$. In particular, $$x_r\in (N_G(x_i)\cup N_G(x_j))\setminus \{x_i,x_j\}.$$Consequently, $A\subseteq (N_G(x_i)\cup N_G(x_j))\setminus \{x_i,x_j\}$. It follows from equality (\ref{17}) that
\begin{align*}
& \big((I^2+ (u_1, \ldots, u_{k-1})):u_k\big)=\big(I^2:u_k\big)+({\rm the\ ideal\ generated\ by}\ A)\\ & =\big(I^2+ ({\rm the\ ideal\ generated\ by}\ A)):u_k\big)\\ & =\big(I(G\setminus A)^2+({\rm the\ ideal\ generated\ by}\ A)):u_k\big)\\ & =\big(I(G\setminus A)^2:u_k\big)+({\rm the\ ideal\ generated\ by}\ A).
\end{align*}
The above equalities imply that$$\depth_S\big((I^2+ (u_1, \ldots, u_{k-1})):u_k\big)=\depth_{S_A}\big(I(G\setminus A)^2:u_k\big),$$where $S_A$ is the polynomial ring obtained from $S$ by deleting the variables in $A$. Using Lemma \ref{last}, we deduce that$$\depth_S((I^2+(u_1, \ldots, u_{k-1})):u_k)\geq \alpha_2(G)-2,$$and if $G$ is a $W(K_3)$-free graph, then$$\depth_S((I^2+(u_1, \ldots, u_{k-1})):u_k)\geq \alpha_2(G)-1.$$This completes the proof of parts (1) and (2).

To prove (3), let $G$ be a triangle-free graph. Then we know from \cite[Lemma 3.10]{rty} that $I(G)^2=I(G)^{(2)}$, where $I(G)^{(2)}$ denotes the second symbolic power of $I(G)$. The assertion now follows from \cite[Theorem 4.2]{s10}.
\end{proof}

The following examples show that the inequalities obtained in Theorem \ref{main} are sharp.

\begin{exmps} \label{sharp}
\begin{itemize}
\item[(1)] Suppose $G=W(K_3)$. Then $\depth_SI(G)^2=1$ which is equal to $\alpha_2(G)-2$.
\item[(2)] Let $G$ be the graph which is obtained from $W(K_3)$ by deleting one of its leaves. Then $\depth_SI(G)^2=1$ which is equal to $\alpha_2(G)-1$.
\item[(3)] Let $G=P_4$ be the path of length three. Then $\depth_SI(G)^2=2$ which is equal to $\alpha_2(G)$.
\end{itemize}
\end{exmps}




\end{document}